\documentclass[10pt]{amsart}
\usepackage[all]{xy}
\usepackage{amsmath}
\usepackage{amsfonts}
\usepackage{amssymb}
\usepackage{amscd}
\usepackage{amsthm}
\usepackage{latexsym}
\usepackage{amsbsy}
\usepackage{color,enumerate}
\usepackage{tikz}
\usetikzlibrary{matrix,arrows,decorations.pathmorphing}
\def\0D{\Delta^{(0)}}
\def\1D{\Delta^{(1)}}

\newcommand{\Fa}{\mathfrak{a}}

\newcommand{\Cc}{\mathcal{C}}
\newcommand{\Tc}{\mathcal{T}}
\newcommand{\mc}{\mathcal{M}}
\newcommand{\Zc}{\mathcal{Z}}

\newcommand{\hmod}{{_H\mathcal{M}}}
\newcommand{\thmod}{{^{\#}_H\mathcal{M}}}
\newcommand{\tihmod}{{^{\#^i}_H\mathcal{M}}}
\newcommand{\tthmod}{{^{**}_H\mathcal{M}}}
\newcommand{\tinvmod}{{^{\#^{-1}}_H\mathcal{M}}}

\newcommand{\id}{1}
\newcommand{\vect}{\text{Vec}}

\newtheorem{theorem}{Theorem}[section]
\newtheorem{remark}[theorem]{Remark}
\newtheorem{proposition}[theorem]{Proposition}
\newtheorem{lemma}[theorem]{Lemma}

\newtheorem{example}[theorem]{Example}
\newtheorem{definition}[theorem]{Definition}

\def\build#1_#2^#3{\mathrel{\mathop{\kern 0pt#1}\limits_{#2}^{#3}}}
\newcommand{\ns}[1]{~\hspace{-3pt}_{\left<#1\right>}}
\newcommand{\ps}[1]{~\hspace{-3pt}^{^{\left(#1\right)}}}

\numberwithin{equation}{section}
\def\ot{\otimes}
\def\part{\partial}

\def\ra{\rightarrow}

\def\text{\hbox}

\def\ot{\otimes}

\def\ra{\rightarrow}

\def\Hom{\mathop{\rm Hom}\nolimits}

\def\Id{\mathop{\rm Id}\nolimits}

\def\build#1_#2^#3{\mathrel{
\mathop{\kern 0pt#1}\limits_{#2}^{#3}}}

\numberwithin{equation}{section}
\newcommand{\comment}[1]{\relax}

\begin{document}

\title{Monoidal Categories,  2-Traces, and Cyclic Cohomology}
\author{Mohammad Hassanzadeh, Masoud Khalkhali, Ilya Shapiro}
\date{}
\maketitle
\begin{abstract}
In this paper we show that to a unital associative algebra object (resp. co-unital co-associative co-algebra object)  of any abelian monoidal category $\Cc$ endowed with  a symmetric $2$-trace, one can attach a cyclic (resp. cocyclic) module, and therefore  speak of the cyclic (co)homology of the (co)algebra ``with coefficients in $F$". We observe that if $\mc$ is a $\Cc$-bimodule category equipped with a stable central pair then $\Cc$ acquires a symmetric 2-trace. The dual notions of symmetric $2$-contratraces and stable central contrapairs are derived as well. As an application we can recover all Hopf cyclic type (co)homology theories, obtain a conceptual understanding of anti-Yetter-Drinfeld modules, and give a formula-free definition of cyclic cohomology.  The machinery can also be applied in settings more general than Hopf algebra modules and comodules.

\end{abstract}

%%%%%%%%%%%%%%%%%%%%%%%%%%%%%%%%%%%%%%%%%%%%%%%%%%%%%%%%%%%%%%%%%%%%%%%%%%%%%%%%%%%%%%%%%%%%%%%%%%%%%%%%%%%%%%%%%%%%%%%%%%%%%%%%%%%%%%%%%%%%%%%%%%%%%%%%%%%%%%%%%%%%%%%%%%%%%%%%%%%%%%%%%%%%%%%%%%%%%%%%%%%%%%%%%%%%%%%%%%%%%%%%%%%%%%%%%%%%%%%%%%%%%%%%%
\section{ Introduction}

One of the major advances   in cyclic cohomology theory in recent years was
the introduction of a new cohomology theory for Hopf algebras by Connes and
Moscovici   and its extension to a cohomology theory for (co)algebras  endowed with  an action of a Hopf algebra and with coefficients in a local system \cite{CM98, hkrs1}.  The local systems  are closely related, and in a sense they are a mirror image of, Yetter-Drinfeld modules over a Hopf algebra.
Beyond Hopf algebras,  one often encounters interesting algebraic objects:  for example quasi Hopf algebras and weak Hopf algebras,  that only possess some of the axioms of Hopf algebras. Developing a Hopf cyclic-type theory for these Hopf-like objects is one of the motivations for this paper.

We find that the  language of monoidal categories is a suitable framework to discuss this question and to discover an answer.
The key point, for us, of the categorical machinery is the notion  of a trace.
Traces in  (monoidal) categories are well-known and have been used in different settings. In \cite{DP} the categorical definition of ordinary traces of square matrices are used for endomorphisms of a dualizable  object in a symmetric monoidal category.   In the derived category of a ring, traces are called Lefschetz numbers. Categorical traces are also used to study fixed-point theory \cite{P}.   The authors in \cite{JSV} showed that traces can be defined for dualizable objects in a monoidal category. On the other hand, if we think of a monoidal category as a bicategory with one object, then the notion of trace that we want is similar to the generalized traces of \cite{PS}.  Finally the authors in \cite{FSS} found a relationship between  the category-valued traces and the twisted center of a monoidal category.

The technical aspects of this paper would have been much more involved and would have required much delicacy, had the notions of bimodule categories and their centers not been already extensively studied.
In particular, as we realized, the center of a certain bimodule category of the monoidal category of (left) modules over a Hopf algebra (sometimes called the twisted center of the monoidal category) provides the suitable coefficients for Hopf cyclic cohomology.  These coefficients were already known, called (stable) anti Yetter-Drinfeld modules \cite{hkrs1}, but were defined much less conceptually. We use these (or dual) central elements  to form a suitable categorical trace to define the desired homological objects.

More precisely, in this paper we use the notion of a {\it symmetric $2$-trace} (compare with the shadow structure in \cite{PS})  for an abelian monoidal category to show that  for a  monoidal category  endowed with such a trace one can attach a cyclic module to any unital associative algebra object. Later we introduce the notion of stable central pair in a monoidal category which is a practical way of obtaining a symmetric $2$-trace.

If $\mc$ is a $\Cc$-bimodule category then so is $Fun(\mc, \vect)$; let   $\Zc_{\Cc}(\mc)$ and $\Zc_{\Cc}Fun(\mc, \vect)$ denote their respective centers. If  $F\in\Zc_{\Cc}Fun(\mc, \vect)$  and $m\in \Zc_{\Cc}(\mc)$ then the pair $(F, m)$ is called a stable central pair if it satisfies one additional mutual compatibility condition given in Definition \ref{stablepair}. We show that any such  pair gives us a symmetric $2$-trace.

As an example we see that the monoidal category $\hmod$ of left modules over a Hopf algebra $H$  can be endowed with a stable central pair and therefore a symmetric $2$-trace. To construct this pair, we consider the $\Cc$ bimodule category $\tinvmod$ where the right $\Cc$-action is given by the monoidal tensor product and the left action is twisted by $S^{-2}$, where $S$ is the antipode of $H$, as we will explain in Section \ref{ydsection}. Then if we start with a central element $M\in \Zc_{\Cc}(\tinvmod)$,  and take $$F:=\Hom_H(1,-)\in \Zc_{\Cc}Fun(\tinvmod, \vect),$$ we observe that  $(F, M)$ is a stable central pair (provided that an extra stability condition is satisfied). More interestingly, we prove that the elements of the center of the $\Cc$-bimodule category $\tinvmod$ are nothing but the ``duals" of the well-known anti Yetter-Drinfeld modules over $H$.  On the other hand, if we pursue a contravariant theory, then the bimodule category of interest is $\thmod$ and we show that $(F,M)$ is a stable central contrapair if $$F:=\Hom_H(-,1),$$ $M$ is the usual anti Yetter Drinfeld module and the stability condition is the usual one \cite{hkrs1}.

We recall that the center of a monoidal category has been studied for different reasons. It is known that the elements of the center of the monoidal category of modules over a Hopf algebra, a weak Hopf algebra, and a quasi Hopf algebra are in fact the Yetter-Drinfeld modules which are the solutions of the quantum Yang-Baxter equations. If we consider this monoidal category as a bimodule category over itself by the left and right actions given by the monoidal tensor product, then the Yetter-Drinfeld modules do indeed form the center of this bimodule category.

One notes that the language of monoidal categories is fundamental  to the  study of Hopf-like objects.   More precisely, for such an object, the category of left modules over it  is a  monoidal category.  Often the axioms that specify the type of the Hopf-like object are themselves dictated by exactly this requirement. There are a great many results about recovering the original Hopf-like object from its associated monoidal category provided that some extra structure (a variation on the fiber functor theme) is provided.  This explains the important relation between  monoidal categories and Hopf algebras. There are also other categorical approaches to cyclic homology \cite{BS}, \cite{KP}. We observe that our categorical machinery can be applied to  the monoidal categories associated to  interesting Hopf-like objects such as weak Hopf algebras, Hopf algebroids, quasi-Hopf algebras and Hopfish algebras to obtain homological constructions such as cyclic homology.\\

\textbf{Acknowledgments}:
The authors would like to thank  the organizers  of the ``Noncommutative Geometry Workshop"  at the University of Western Ontario, June 2015,  where  this paper began.
\bigskip

%%%%%%%%%%%%%%%%%%%%%%%%%%%%%%%%%%%%%%%%%%%%%%%%%%%%%%%%%%%%%%%%%%%%%%%%%%%%%%%%%%%%%%%%%%%%%%%%%%%%%%%%%%%%%%%%%%%%%%%%%%%%%%%%%%%%%%%%%%%%%%%%%%%%%%%%%%%%%%%%%%%%%%%%%%%%%%%%
%\tableofcontents
%%%%%%%%%%%%%%%%%%%%%%%%%%%%%%%
\subsection{Motivation}\label{motivation}

 The authors in \cite{hkrs1} introduced Hopf cyclic cohomology with coefficients for four  types of symmetries. In the case of a $H$-module coalgebra $C$, for a right-left stable anti Yetter-Drinfeld module (SAYD) over $H$, they assign a cocyclic module structure to $C^n=\Hom_H(k, M\ot C^{\ot n+1})$.  This theory generalizes Connes-Moscovici's  Hopf cyclic cohomology theory \cite{CM98}.  In the case of a $H$-module algebra $A$, and also for a right-left stable anti Yetter-Drinfeld module (SAYD) over $H$, they assign a cocyclic module structure to $C^n=\Hom_H(M\ot A^{\ot n+1}, k)$.
 The comodule part of the anti-Yetter-Drinfeld (AYD) module structure appears in the cyclic map $\tau$. The mysterious AYD structure has not been conceptually well-understood in the literature, although it is known that this structure is obtained by replacing the antipode $S$ by $S^{-1}$ in the definition of a Yetter-Drinfeld (YD) module. On the other hand, the YD modules are well-understood as they form the center of the monoidal category of $H$-modules, $\hmod$. Not only was the categorical meaning of AYD modules not understood, but also it was not clear   why such a mysterious structure is needed to obtain a cocyclic module and therefore cyclic cohomology.
\subsubsection{Contravariant cohomology theory.}
To answer this question, one can start from scratch and try to define the cyclic map $\tau$ on $C^n=\Hom_H(M\ot A^{\ot n+1}, k)$ directly.  This is the only significant addition to the already apparent cosimplicial structure (in the case that $A$ is an algebra, unital and associative of course).  More precisely, we need to \emph{slide the first copy of $A$ past $M$ and then to the back.} To understand the idea better let us consider a special case when the monoidal category is rigid, such is the category of \emph{finite} dimensional left modules over a Hopf algebra $H$. Later we will see that the finiteness assumption can be removed. Using the standard adjunction properties of rigidity, for any $V, W\in \hmod$, we have:

 $$\Hom_{H}(V\ot W, 1)\simeq \Hom_{H}(V, 1\ot {W^*})\simeq \Hom_{H}(V, {W^*}\ot 1)\simeq \Hom_H({W^{**}}\ot V, 1).$$

Thus for an $M\in \hmod$ and an algebra object $A$ (though the algebra structure plays no role at this stage) in $\hmod$ we obtain

\begin{equation}\label{adj22} \Hom_{H}(M\ot A^{\ot n+1}, 1)\simeq  \Hom_{H}({A^{**}}\ot M\ot A^{\ot n}, 1). \end{equation}

If we have an extra condition on $M$, namely that  \begin{equation}\label{centre22}{A^{**}}\ot M   \simeq  M\ot A ,\end{equation} then we obtain the desired $\tau$, i.e.,
$$\tau_n:\Hom_{H}(M\ot A^{\ot n+1}, 1)\simeq \Hom_{H}({A^{**}}\ot M\ot A^{\ot n}, 1) \simeq \Hom_{H}(M\ot A^{\ot n}\ot A,1), $$ where we first use \eqref{centre22}, followed by the inverse of \eqref{adj22}.
This suggests that for a rigid monoidal category $\Cc$
$$AYD(\Cc)=\{M \in \Cc : {A^{**}}\ot M  \xrightarrow{\sim}   M\ot A , \, \forall A\in \Cc\},$$
with some compatibility conditions.

Knowing  that $YD$ modules form the center of the  monoidal category, i.e., $YD=\mathcal{Z}(\Cc)$, we see that  $AYD$  is to $YD$ as ${A^{**}}\ot M \xrightarrow{\sim} M\ot A$ is to $A\ot  M\xrightarrow{\sim} M\ot A$. The stability condition in the case of the usual stable anti Yetter-Drinfeld  modules ensures that $\tau^{n+1}_n=\Id.$ To obtain the same conclusion in our general case leads us to the requirement that the single cyclic map $\tau_0$ be the identity. In summary, from the above considerations, we guess that whereas $YD= \mathcal{Z}( \hmod_{fd})$ (center of a monoidal category), $AYD=\mathcal{Z}({\tthmod}_{fd})$ (center of a bimodule category) where ${\tthmod}_{fd}$ is simply $\hmod_{fd}$ with the left action  modified by  ${(-)^{**}}$. This guess turns out to be correct.  Note that both $\tau_n$ and $$AYD(\Cc)=\mathcal{Z}_{\Cc}({^{**}\Cc})$$ make sense for any rigid category.
\subsubsection{Covariant cohomology theory.}
If instead we consider $C^n=\Hom_H(k, M\ot C^{\ot n+1})$ and try to define $\tau$ directly, we need to \emph{slide the first copy of $C$ past $M$ and then to the back.} Again let us consider the \emph{finite} dimensional left modules over a Hopf algebra $H$. Using the standard adjunction properties of rigidity, for any $V, W\in \hmod$, we have:

 $$\Hom_{H}(1, V\ot W)\simeq \Hom_{H}( 1\ot {^* W}, V)\simeq \Hom_{H}( {^* W}\ot 1, V)\simeq \Hom_H(1, {^{**}W}\ot V).$$
 
Thus for an $M\in \hmod$ and a coalgebra object $C$ (though the coalgebra structure plays no role at this stage) in $\hmod$ we obtain

\begin{equation}\label{adj2} \Hom_{H}(1, M\ot C^{\ot n+1})\simeq  \Hom_{H}(1, {^{**}C}\ot M\ot C^{\ot n}). \end{equation}

If we have an extra condition on $M$, namely that  \begin{equation}\label{centre2}{^{**}C}\ot M   \simeq  M\ot C ,\end{equation} then we obtain the desired $\tau$, i.e.,
$$\tau_n:\Hom_{H}(1, M\ot C^{\ot n+1})\simeq \Hom_{H}(1, {^{**}C}\ot M\ot C^{\ot n}) \simeq \Hom_{H}(1, M\ot C^{\ot n}\ot C), $$ where we first use the inverse of \eqref{centre2} followed by the inverse of \eqref{adj2}.
This suggests that for a rigid monoidal category $\Cc$ we also need
$$YD_1(\Cc)=\{M \in \Cc : {^{**}C}\ot M  \xrightarrow{\sim}   M\ot C , \, \forall C\in \Cc\},$$
with some compatibility conditions, to serve as coefficients.  Note that $YD_1$ is not the same as $AYD$, it is ``dual" to it.

\section{Preliminaries}\label{prelim}
Here we collect some background material that should facilitate the reading of this paper.  The content of Sections \ref{ydsection} and \ref{stabilitysection} is new.  The discussion involving Yetter-Drinfeld modules, anti-Yetter-Drinfeld modules, and their generalizations contained in Section \ref{ydsection} is especially important.  The conceptual reinterpretation of these objects and their associated complicated formulas was one of the motivations for this paper.

\subsection{(Co)cyclic modules.}
The main goal  of this  paper is to introduce a suitable categorical language to unify different notions of  cyclic homology under  a single theory. Therefore we need to  recall the definitions of cyclic and cocyclic modules from \cite{Connes-cyclic} and  \cite{loday}.

Recall that  the simplicial category $\Delta$ has as its objects non-negative integers considered as totally ordered sets $[n]=\{0,1,\cdots,n\}$ and its morphisms are non-decreasing functions $[n]\to[m]$.  A simplicial module is a  contravariant functor from $\Delta$ to $\vect$.  Similarly, a cosimplicial module is a covariant functor.  By keeping the same objects and adding cyclic permutations we obtain Connes cyclic category $C$.  A cyclic module is again a contravariant functor from $C$ to $\vect$, while a cocyclic module is a covariant one.  

More explicitly, a {\em cosimplicial module} is given by the data $(C_{n},\delta_{i}, \sigma_{i})$ where $\{C_{n}\}$, $n\geq 0$ is a sequence of vector spaces over the field $k$. The maps $\delta_{i}:C^{n}\rightarrow C^{n+1}$ are called cofaces, and $\sigma_{i}:C^{n}\rightarrow C^{n-1}$ are called codegeneracies. These are $k$-linear maps satisfying the following cosimplicial relations:
\begin{align}\label{rel1}
\begin{split}
\delta_{j}  \delta_{i} &= \delta_{i} \delta_{j-1},  \quad   i <j,\\
\sigma_{j}  \sigma_{i} &= \sigma_{i} \sigma_{j+1},    \quad i \leq j,\\
\sigma_{j} \delta_{i} &=
 \begin{cases}
\delta_{i} \sigma_{j-1},   &
 i<j\\
\text{Id},   &   i=j \,\,\text{or}\,\, i=j+1,\\
\delta_{i-1} \sigma_{j},  & i>j+1.
\end{cases}
\end{split}
\end{align}

A {\em cocyclic module }  is a cosimplicial module equipped with the extra morphisms $\tau_n :C^n\rightarrow C^n$, called cocyclic maps such that the following relations hold:
\begin{align}\label{rel2}
\begin{split}
\tau_{n}\delta_{i}&=\delta_{i-1} \tau_{n-1}, \quad 1\le i\le  n,\\
\tau_{n} \delta_{0}&=\delta_{n},\\
\tau_n \sigma_{i}&=\sigma_{i-1} \tau_{n+1}, \quad 1\le i\le n,\\
\tau_n^{n+1}&=\Id.\\
\end{split}
\end{align}

In a dual manner, one can define a cyclic module as a simplicial module with extra cyclic maps. More precisely,
 a cyclic module is given by the data  $(C_{n},\delta_i, \sigma_{i}, \tau_{n})$, where $C_{n}$, $n\geq 0$ is a $k$-vector space and  $\delta_{i}: C_{n}\rightarrow C_{n-1},  \quad \sigma_{i}:C_n \ra C_{n+1},  \quad 0 \leq i \leq n$, and  $\tau_n:C_n \ra C_n$, are called faces, degeneracies and cyclic maps respectively. These are $k$-linear maps, satisfying the following  relations:
\begin{align}\label{rel11}
\begin{split}
\delta_{i} \delta_{j}&=  \delta_{j-1} \delta_{i},  \quad i <j,\\
\sigma_{i}  \sigma_{j}&= \sigma_{j+1} \sigma_{i}, \quad i \leq j,\\
\delta_{i} \sigma_{j} &=
 \begin{cases}
\sigma_{j-1} \delta_{i},   & i<j,\\
\Id,    & i=j \,\,\text{or}\,\, i=j+1,\\
\sigma_{j} \delta_{i-1},  & i>j+1.
\end{cases}
\end{split}
\end{align}
and
\begin{align}\label{rel22}
\begin{split}
\delta_{i}\tau_n&= \tau_{n-1}\delta_{i-1}, \quad 1\le i\le n,\\
\delta_{0} \tau_n&= \delta_{n}\\
\sigma_{i} \tau_n&= \tau_{n+1}  \sigma _{i-1},\quad 1\le i\le n,\\
\tau_n^{n+1}&= \Id. \\
\end{split}
\end{align}
One notes that the relation $\sigma_{0} \tau_n = \tau_{n+1}^2  \sigma_{n}$ that is usually listed along with the above is an extra relation \cite[section 5.2]{loday} which can be obtained from  $\tau^{n+1}_n=\Id$ and $\sigma_{i} \tau_n= \tau_{n+1}  \sigma _{i-1}$. Similarly for a cocyclic module $\tau_n \sigma_{0} = \sigma_{n} \tau_{n+1}^2$ can be obtained from  the other relations.
From a (co)cyclic module, one can define  Hochschild, cyclic and periodic cyclic (co)homology \cite{loday}.

\subsection{$H$-modules, $H$-comodules, and compatibility conditions.}\label{ydsection}
Recall that the center of the monoidal category  of left $H$-modules, $\hmod$,  is equivalent to the category of left-right  Yetter-Drinfeld modules. For finite dimensional Hopf algebras  the center $\mathcal{Z}(\hmod)$ is also equivalent to the representations of the quantum double $_{D(H)}\mathcal{M}$. For details, we refer the reader to \cite{kassel}.
We recall from \cite{majid,sch} that for a Hopf algebra (or a bialgebra) $H$, a left $H$-module, right $H$-comodule $M$ is called a Yetter-Drinfeld module if
\begin{equation}\label{YD1}
  h\ps{1}m\ns{0}\ot h\ps{2} m\ns{1}=(h\ps{2} m)\ns{0}\ot (h\ps{2} m)\ns{1}h\ps{1}.
\end{equation}
For Hopf algebras with an invertible antipode this is equivalent to
\begin{equation}\label{YD2}
  (hm)\ns{0}\ot (hm)\ns{1}= h\ps{2}m\ns{0}\ot h\ps{3}m\ns{1}S^{-1}(h\ps{1}).
\end{equation}
The isomorphism giving the central structure of a left right YD module $M$ is \begin{align*}V\ot M&\xrightarrow\sim M\ot V\\v\ot m&\mapsto m\ns{0}\ot m\ns{1}v\end{align*} and the YD condition simply ensures that the map above is that of $H$-modules.

We recall  from \cite{hkrs2}  that   a left-right anti Yetter-Drinfeld module $M$ over a Hopf algebra $H$   is  a left $H$-module and a right $H$-comodule satisfying
 \begin{equation}\label{AYD1}
   (hm)\ns{0}\ot (hm)\ns{1}=  h\ps{2}m\ns{0}\ot h\ps{3}m\ns{1}S(h\ps{1}).
 \end{equation}
We denote the category of left-right AYD modules over a Hopf algebra $H$ by  $_H\mathcal{AYD}^H$.  Note that there are three additional flavors of AYD modules: left-left, right-left, right-right.  All of them are equivalent and thus we focus only on the left-right variety.  We will need to generalize.

\begin{definition}
Let $M$ be a left module and a right comodule over $H$, and let $i\in\mathbb{Z}$.  We say that $M\in {_H\mathcal{YD}}^H_i$ if \begin{equation}\label{genyd1}(h\ps{2}m)\ns{0}\ot (h\ps{2}m)\ns{1} S^{-2i}(h\ps{1})=h\ps{1}m\ns{0}\ot h\ps{2}m\ns{1}.\end{equation}  We say that $M$ is a generalized Yetter-Drinfeld module.
\end{definition}

The following  lemma provides a characterization of the generalized Yetter-Drinfeld modules akin to the one above for the Yetter-Drinfeld modules.

\begin{lemma}
For a Hopf algebra $H$ the generalized $i$th YD condition \eqref{genyd1} is  equivalent to
\begin{equation}\label{genyd2}
  \rho(hm)=(hm)\ns{0}\ot(hm)\ns{1}=h\ps{2}m\ns{0}\ot h\ps{3}m\ns{1} S^{-1-2i}(h\ps{1}).
  \end{equation}
 \end{lemma}

Thus ${_H\mathcal{YD}}^H_0={_H\mathcal{YD}}^H$, while ${_H\mathcal{YD}}^H_{-1}={_H\mathcal{AYD}}^H$, and we will also need ${_H\mathcal{YD}}^H_1$ to serve as coefficients of the covariant theory.  Note that if instead of $\hmod$ we consider $\hmod_{fd}$ then $\mathcal{YD}_1=\mathcal{AYD}^*=\mathcal{AYD}\, contramodules$.

 \begin{proof}
 First, we show that $\eqref{genyd2}\Rightarrow\eqref{genyd1}$:
 \begin{align*}
   (h\ps{2}m)\ns{0}\ot (h\ps{2}m)\ns{1} S^{-2i}(h\ps{1})  &=h\ps{3}m\ns{0}\ot h\ps{4} m\ns{1}S^{-1-2i}(h\ps{2})S^{-2i}(h\ps{1})\\
   %&=h\ps{-1}m\ns{-1}S^{-1}(S^{-1}(h\ps{4})h\ps{3})\ot h\ps{2}m\ns{0}\\
   &=h\ps{2}m\ns{0}\ot h\ps{3}m\ns{1}\varepsilon(h\ps{1})\\
   &=h\ps{1}m\ns{0}\ot h\ps{2}m\ns{1}.
 \end{align*}
 Now we show that $\eqref{genyd1}\Rightarrow\eqref{genyd2}$:
 \begin{align*}
   h\ps{2}m\ns{0}\ot h\ps{3}m\ns{1}S^{-1-2i}(h\ps{1})   &=(h\ps{3}m)\ns{0}\ot (h\ps{3}m)\ns{1}S^{-2i}(h\ps{2})S^{-1-2i}(h\ps{1})\\
   %&=(h\ps{1}m)\ns{-1}S^{-1}(h\ps{3}S^{-1}(h\ps{2}))\ot (h\ps{1}m)\ns{0}\\
   &=(h\ps{2}m)\ns{0}\ot (h\ps{2}m)\ns{1}\varepsilon(h\ps{1})\\
   &=\varepsilon(h\ps{1})\rho(h\ps{2}m)\\
   &=\rho(\varepsilon(h\ps{1})h\ps{2}m)\\
   &=\rho(hm).
 \end{align*}

 \end{proof}

If $\Cc$ is a monoidal category, $\mc$ a $\Cc$-bimodule category, and $F: \mc\rightarrow \mc$ a monoidal endofunctor, then we use $\mc^F$ and ${^F\mc}$ to denote the bimodule categories with the right and respectively left  actions   twisted by $F$. More precisely, for $V,W\in\Cc$ and $M\in{^F\mc}$ we have $$V\cdot_{new} M\cdot_{new} W=F(V)\cdot_{old} M\cdot_{old} W,$$ with $\mc^F$ defined analogously.  Note that if $F$ is an equivalence then ${^F\mc}\simeq\mc^{F^{-1}}$.

Let  $\#: \hmod \rightarrow \hmod$ be the functor which takes a left $H$-module $M$ to  $M^\#\in \hmod$ where $M^\#$ is the same as $M$ as a vector space but the left action is modified by $S^{2}$, i.e., is now given by $h\cdot m= S^{2}(h)m$.  If, as we always assume, $S$ is invertible, then $\#$ is an autoequivalence of $\hmod$.  Thus for $i\in\mathbb{Z}$ we can consider $\tihmod$, i.e., $$V\cdot M\cdot W=V^{\#^i}\ot M\ot W.$$

We can now repeat verbatim the same arguments  as in the usual, YD modules Vs center, case.  Roughly speaking, let $M$ be in $\Zc_{\hmod}\tihmod$, then $M$ is already a left $H$ module and for every $V\in\hmod$ we have an isomorphism $\Phi:V^{\#^i}\ot M\xrightarrow\sim M\ot V$.  Take $V=H$ and define the right comodule structure on $M$ via $$\rho(m)=\Phi(1\ot m).$$  Conversely, suppose that $M$ is in ${_H\mathcal{YD}}^H_{-i}$.  Then for every $V\in\hmod$ define $\Phi$ by \begin{equation}\label{mcent}\Phi(v\ot m)=m\ns{0}\ot m\ns{1}v.\end{equation}  Note that the requirement that $\Phi$ be an $H$ module map is exactly the equation \eqref{genyd1}.  Furthermore, $$\Phi^{-1}(m\ot v)=S(m\ns{1})v\ot m\ns{0}.$$  We have arrived at the following theorem:

\begin{theorem}\label{ydcenter}
For a Hopf algebra $H$ with an invertible antipode and $i\in\mathbb{Z}$,

$$\Zc_{\hmod}\tihmod\simeq {_H\mathcal{YD}}^H_{-i}.$$

\end{theorem}

\begin{remark}
If we consider the action of  $\mathbb{Z}$ on $\hmod$ via $\#$, then $\hmod\rtimes\mathbb{Z}$ is a $\mathbb{Z}$ graded monoidal category.  If we write $\hmod\rtimes\mathbb{Z}=\bigoplus_i\mc_i$ then $\mc_i=\hmod^{\#^i}$ as an $\hmod$ bimodule category and $\bigoplus_i{_H\mathcal{YD}}^H_{i}$ is a $\mathbb{Z}$-equivariant $\mathbb{Z}$-braided monoidal category.
\end{remark}

\subsection{Stability.}\label{stabilitysection}

Recall that a left $H$ module and right $H$ comodule $M$ is called stable if $m\ns{1}m\ns{0}=m$.  We will need a slightly more general notion for the covariant theory.  The classical stability will be precisely correct for the contravariant version.

\begin{definition}
Let $i\in\mathbb{Z}$. A left $H$ module and right $H$ comodule $M$ is called $i$-stable if $S^{2i}(m\ns{1})m\ns{0}=m$.
\end{definition}

Thus the usual stability is now $0$-stability.  The following lemma shows what happens to the odd powers of the antipode.

\begin{lemma}\label{stability}
For a left $H$ module and right $H$ comodule $M$ we have $$S^{2i}(m\ns{1})m\ns{0}=m\quad\Leftrightarrow\quad S^{2i-1}(m\ns{1})m\ns{0}=m.$$
\end{lemma}

\begin{proof}
This is a direct computation, however in the instances where we see it, a more conceptual explanation can be found in terms of the $\tau_0$ map and its inverse that play a key role in our more conceptual understanding of stability.

``$\Rightarrow$"
\begin{align*}
S^{2i-1}(m\ns{1})m\ns{0}&=S^{2i-1}(m\ns{2})S^{2i}(m\ns{1})m\ns{0}\\
&=S^{2i-1}(S(m\ns{1})m\ns{2})m\ns{0}\\
&=\varepsilon(m\ns{1})m\ns{0}\\
&=m.
\end{align*}

``$\Leftarrow$"
\begin{align*}
S^{2i}(m\ns{1})m\ns{0}&=S^{2i}(m\ns{2})S^{2i-1}(m\ns{1})m\ns{0}\\
&=S^{2i-1}(m\ns{1}S(m\ns{2}))m\ns{0}\\
&=\varepsilon(m\ns{1})m\ns{0}\\
&=m.
\end{align*}
\end{proof}

\section{Monoidal categories  and 2-traces}\label{moncat2tr}
This section develops the core of the conceptual machinery that we need in order to understand the Hopf-type cyclic homology theories.  For convenience we start with the covariant case and derive the contravariant case from it.  We note that ignoring the non-strictness of the monoidal category by suppressing the explicit formulas for associators would have cleaned up the exposition.  Our choice to include them was motivated by future applications of this machinery to monoidal categories where the associator appears as an explicit formula and so would have to appear in the definitions of the cyclic structure once it is unpacked from the conceptual definitions. The ease with which such laborious formulas are safely hidden from view demonstrates the power of the categorical machinery.

Let $(\mathcal{C}, \ot)$ be a monoidal category. We will need the following conventions.  Let $A$ be an object in $\Cc$, by $A^{\ot n}$ we mean an object defined inductively as $$A^{\ot n}=A^{\ot n-1}\ot A.$$  For $\vec{n}=(n_1,\cdots, n_k)$ with $n_i$ non-negative integers, by $A^{\ot \vec{n}}$ we denote an object defined inductively as $$A^{\ot \vec{n}}=A^{\ot (n_1,\cdots, n_{k-1})}\ot A^{\ot n_k}.$$  We interpret $A^{\ot 0}$ as the unit object  $\id$.  We apply a similar convention to morphisms.  Let $|\vec{n}|=n_1+\cdots+n_k.$ Then for $\vec{n}$ and $\vec{m}$ with $|\vec{n}|=|\vec{m}|$ denote by $$\Fa_{\vec{n}}^{\vec{m}}: A^{\ot\vec{n}}\rightarrow A^{\ot\vec{m}}$$ the unique isomorphism ensured by the monoidal structure.  Omitting the brackets enclosing the vector components to reduce clutter, we thus have $\Fa_{n+1}^{n,1}=Id$, while $\Fa_{n,1}^{1,n}$ is in general highly non-trivial  and will play a central role below.  Though $\Fa_{n,1}^{1,n}$ is invisible for Hopf algebras, it will be needed for quasi-Hopf algebras and similar objects which lack ``on the nose" coassociativity.
\subsection{Symmetric 2-traces}
Let $\vect$ be the  category of vector spaces, and $\mathcal{M}$ be a $\Cc$-bimodule category. Then the functor category $Fun(\mathcal{M}, \vect)$ is a $\mathcal{C}$-bimodule category with the left and right actions defined by

\begin{equation}
  c\cdot F(-):= F(-\cdot c), \quad \text{and} \quad F\cdot c(-):= F(c\cdot -)
\end{equation}
for all $c\in \Cc$.  The center of a $\Cc$-bimodule category $\mathcal{M}$ is denoted by $\Zc_{\Cc}\mathcal{M}$. Since $\Cc$ is a $\Cc$-bimodule category using its tensor product, we can set $\mathcal{M}=\Cc$. To simplify the notation the center of a monoidal category  $\Cc$ will be denoted by $\Zc(\Cc)$.

\begin{definition}
Let $(\mathcal{C}, \ot)$ be a monoidal category.
\begin{itemize}
\item A functor
$$F\in \mathcal{Z}_{\mathcal{\Cc}}Fun(\mathcal{C}, \vect),$$ is called a 2-trace.  In particular we have natural isomorphisms
\begin{equation}
  \iota_c(-): F(-\ot c)  \rightarrow F(c\ot -).
\end{equation}

\item  A  2-trace $F$ is called a symmetric 2-trace (compare  with  the shadow structure in \cite{PS})  if
\begin{equation}
  \iota_{c}(1)=F(\Fa_{0,1}^{1,0}).
\end{equation}
\end{itemize}
\end{definition}

Note that the symmetry condition is indeed worthy of its name as it ensures that $$\iota_c(c')\iota_{c'}(c)=Id_{F(c\ot c')}.$$

\begin{example}\rm{
Let $A$ be an associative algebra and let $\Cc=\text{Bimod}(A)$ denote the tensor category of $A$-bimodules.  Then an example of a symmetric $2$-trace is provided by  the functor $HH_0(A,-),$ the $0$th Hochschild homology of an $A$-bimodule \cite{FSS}.  We note that for our purposes as outlined below, this example is not very interesting, its only advantage is that it is easy to explain.
}

\end{example}
Recall that we denote the subcategory of unital associative algebras in $\Cc$ by $Alg(\Cc)$. We denote the multiplication morphism of  an algebra object  $A\in \Cc$ by $m: A\ot A\rightarrow A$ and its  unit morphism by $u: 1\rightarrow A$.
Given an algebra  $A\in Alg(\Cc)$ and a symmetric  2-trace $F: \Cc \to \vect$, we define a cyclic
object in $\vect$ as follows.

\begin{definition}
Let $$C_n(A)=F(A^{\ot n+1}), n \geq 0.$$ We define the cyclic structure on $C_n (A)$ by
\begin{itemize}
\item    $\tau_n = F(\Fa_{1,n}^{n,1})\circ\iota_A(A^{\ot n})$,
\item     $\delta_i = F(\Fa^{n-1,1}_{i,1,n-i-1})\circ F(Id^{\ot i}\ot m\ot \Id^{\ot n-i-1})\circ F(\Fa_{n,1}^{i,2,n-i-1})$, for $0 \leq i \leq n-1$,
\item     $\delta_{n} =  \delta_0\circ\tau_n$,
\item     $\sigma_i = F(\Fa^{n+1,1}_{i+1,1,n-i})\circ F(\Id^{\ot i+1}\ot u\ot \Id^{\ot n-i})\circ F(\Fa_{n,1}^{i+1,0,n-i})$, for $0 \leq i \leq n$.
\end{itemize}
\end{definition}

Note that for $0 \leq i \leq n-1$ we have $$\delta_i=F(\delta_i^{(n)})$$ and for $0 \leq i \leq n-2$ we have $$\delta_i=F(\delta_i^{(n-1)}\ot Id).$$  Similarly, for $0 \leq i \leq n$ we have $$\sigma_i=F(\sigma_i^{(n)})=F(\Fa_{1,n+1}^{n+1,1})F(Id\ot \sigma_{i-1}^{(n-1)})F(\Fa_{n,1}^{1,n}),$$ and observe that $\sigma_{-1}$ makes sense and is useful.  These observations become relevant in the following Proposition.

\begin{proposition}\label{trace1}
  For any $A\in Alg(\Cc)$  and  any symmetric 2-trace $F: \Cc \to \vect$ we have a cyclic object $C_{\bullet}(A)=F(A^{\ot \bullet +1})$ in $\vect$.
\end{proposition}
\begin{proof}
To see the  simplicial relations we apply the functor $F$ to the simplicial relations that are classically satisfied by $\delta_i^{(n)}$'s and $\sigma_i^{(n)}$'s, with the exception of those involving the special $\delta_n$.  One can check that the latter all follow formally from the former simplicity relations and the cyclicity relations below.

Here we check the cyclicity relations. First for any $1\leq i \leq n$, we show that $\delta_{i}\tau_n= \tau_{n-1} \delta_{i-1}$. We begin with the case $1\leq i \leq n-1$:
\begin{align*}
  \tau_{n-1} \delta_{i-1} &= F(\Fa_{1,n-1}^{n-1,1})\iota_A(A^{\ot n-1})F(\delta_{i-1}^{(n-1)}\ot Id)\\
   &= F(\Fa_{1,n-1}^{n-1,1}) F(Id\ot \delta_{i-1}^{(n-1)})\iota_A(A^{\ot n})\\
   &= F(\delta_i^{(n)})F(\Fa_{1,n}^{n,1})\iota_A(A^{\ot n})\\
   &= \delta_i \tau_n.
\end{align*}
For $i=n$ we observe that $\delta_n\tau_n=\tau_{n-1}\delta_{n-1}$ iff $\delta_0\tau^2_n=\tau_{n-1}\delta_{n-1}$, since $\delta_n=\delta_0\tau_n$ by definition, and so:
\begin{align*}
  \tau_{n-1} \delta_{n-1} &= F(\Fa_{1,n-1}^{n-1,1})\iota_A(A^{\ot n-1})F(Id^{\ot n-1}\ot m)F(\Fa_{n,1}^{n-1,2})\\
   &= F(\Fa_{1,n-1}^{n-1,1}) F(m\ot Id^{\ot n-1})\iota_{A^{\ot 2}}(A^{\ot n-1})F(\Fa_{n,1}^{n-1,2})\\
   &= F(\Fa_{1,n-1}^{n-1,1}) F(m\ot Id^{\ot n-1})F(\Fa_{n,1}^{2,n-1})F(\Fa^{n,1}_{2,n-1})\iota_{A^{\ot 2}}(A^{\ot n-1})F(\Fa_{n,1}^{n-1,2})\\
   &= \delta_0 \tau_n^2.
\end{align*}

Here we show that $\sigma_i \tau_n=\tau_{n+1} \sigma_{i-1}$ for all $0\leq i\leq n$:
 \begin{align*}
   \sigma_i \tau_n &= F(\Fa_{1,n+1}^{n+1,1})F(Id\ot \sigma_{i-1}^{(n-1)})F(\Fa_{n,1}^{1,n})F(\Fa^{n,1}_{1,n})\iota_A(A^{\ot n})\\
   &= F(\Fa_{1,n+1}^{n+1,1})F(Id\ot \sigma_{i-1}^{(n-1)})\iota_A(A^{\ot n})\\
   &= F(\Fa_{1,n+1}^{n+1,1})\iota_A(A^{\ot n+1})F(\sigma_{i-1}^{(n-1)}\ot Id)\\
   &= \tau_{n+1} \sigma_{i-1}.
 \end{align*}

Finally we demonstrate that $\tau_n^{n+1}=Id$:
 \begin{align*}
   \tau^{n+1}_n &= F(\Fa_{n+1,0}^{n,1})\iota_{A^{\ot n+1}}(1)F(\Fa_{n,1}^{0,n+1})\\
   &= F(\Fa_{n+1,0}^{n,1})F(\Fa_{0,n+1}^{n+1,0})F(\Fa_{n,1}^{0,n+1})\\
   &= F(\Fa_{n,1}^{n,1})\\
   &= Id.
 \end{align*}

\end{proof}

Note that there was nothing special about $\vect$ in the above considerations, namely the results would still hold if $\vect$ was replaced by any target category $\Tc$, namely a symmetric $\Tc$-valued $2$-trace would still produce cyclic objects in $\Tc$ from elements of $Alg(\Cc)$.

Let $\Cc_{op}$ denote the opposite monoidal category \emph{with only the arrows reversed}. Thus the associator is replaced by its inverse. Let $\mc$ be a $\Cc$ bimodule category, then $\mc_{op}$ is a $\Cc_{op}$ bimodule category via $$c'\cdot m'\cdot d'=(c\cdot m\cdot d)',$$ where we use $m'$ to denote the element $m\in\mc$ when we consider it as an element of $\mc_{op}$. Recall that for $m\in\Zc_\Cc(\mc)$ we have isomorphisms $\iota^m_c:c\cdot m\xrightarrow\sim m\cdot c$.  We note that \begin{align*}\Zc_\Cc\mc&\simeq\Zc_{\Cc_{op}}\mc_{op}\\m&\mapsto m'\\\iota^m_c&\mapsto \iota'^{m'}_{c'}=((\iota^m_c)^{-1})'.\end{align*}

Consider $\Tc=\vect_{op}$ and replace $\Cc$ by $\Cc_{op}$.  More precisely, let $Coalg(\Cc)$ denote the subcategory of coassociative counital coalgebra objects of $(\Cc, \ot)$. Then \begin{align*}Fun(\Cc,\vect)&=Fun(\Cc_{op},\vect_{op})_{op},\\
\mathcal{Z}_{\Cc}Fun(\Cc,\vect)&\simeq\mathcal{Z}_{\Cc_{op}}Fun(\Cc,\vect)_{op}\\&=\mathcal{Z}_{\Cc_{op}}Fun(\Cc_{op},\vect_{op}),\\
Coalg(\Cc)&=Alg(\Cc_{op})_{op}.
\end{align*}  Furthermore, a cyclic object in $\vect_{op}$ is the same as a cocyclic object in $\vect$ and we have arrived at the following:

\begin{proposition}\label{trace2}
  If $C\in Coalg(\Cc)$ and $F$ a symmetric 2-trace then $C^\bullet= F(C^{\ot \bullet+1})$ is a cocyclic object in $\vect$.
\end{proposition}

\begin{remark}
Recall that for an algebra $A$, we had $$\tau_n = F(\Fa_{1,n}^{n,1})\circ\iota_A(A^{\ot n}).$$ However, after unraveling the above identifications, we have for a coalgebra $C$: $$\tau_n = \iota^{-1}_C(C^{\ot n})\circ F(\Fa_{n,1}^{1,n}).$$

\end{remark}

\subsubsection{The contravariant functor $F$.}
While the covariant theory discussed above is suitable for explaining the cocyclic structure for the case  $C^n=\Hom_H(k, M\ot C^{\ot n+1})$.  If we want to deal with the case of $C^n=\Hom_H(M\ot A^{\ot n+1}, k)$ and obtain a cocyclic structure on it, then we need a contravariant $F$.  This is not hard to do in light of the above.

\begin{definition}
We say that a contravariant functor $F$ from $\Cc$ to $\vect$ is a symmetric $2$-contratrace if $F$ is a symmetric $2$-trace on $\Cc_{op}$.
\end{definition}
By recalling that $Coalg(\Cc)=Alg(\Cc_{op})_{op}$ we immediately obtain the following:
\begin{proposition}\label{trace3}
  If $C\in Coalg(\Cc)$ and $F$ a symmetric 2-contratrace then $C^\bullet= F(C^{\ot \bullet+1})$ is a cyclic object in $\vect$.
\end{proposition}

While $Alg(\Cc)=Coalg(\Cc_{op})_{op}$ implies that:

\begin{proposition}\label{trace4}
  If $A\in Alg(\Cc)$ and $F$ a symmetric 2-contratrace then $C^\bullet= F(A^{\ot \bullet+1})$ is a cocyclic object in $\vect$.
\end{proposition}

\begin{remark}
Now for an algebra $A$, we have $$\tau_n = \iota^{-1}_A(A^{\ot n})\circ F(\Fa^{n,1}_{1,n}).$$ Furthermore, for a coalgebra $C$ we get: $$\tau_n = F(\Fa^{1,n}_{n,1})\circ\iota_C(C^{\ot n}).$$

\end{remark}

\subsection{Stable central pairs}
The concept of a stable central pair introduced in the following definition arises naturally in settings generalizing the Hopf-cyclic theory.  The Hopf-cyclic theory itself is implicitly based on it.  As the Lemma below demonstrates the reason for its usefulness is that it is a natural way of constructing symmetric $2$-traces, which lead, as we saw above, to cyclic objects.

\begin{definition}\label{stablepair}
  Let $(\mathcal{C}, \ot)$ be a monoidal category,  and  $\mathcal{M}$ a $\Cc$-bimodule category.
    Let $F\in Fun (\mathcal{M}, \vect)$ and $m\in \mathcal{M}$.    The pair $(F, m)$ is called a central pair if
  \begin{itemize}
    \item  $F\in \mathcal{Z}_{\Cc} Fun(\mathcal{M}, \vect)$, in particular $\iota^F_c(-): F(-\cdot c)\simeq F(c\cdot -)$.
    \item  $m\in \mathcal{Z}_{\Cc}(\mathcal{M})$, in particular $\iota^m_c: c\cdot m \simeq m\cdot c$.
    \  \end{itemize}
The central pair $(F, m)$ is called a stable central pair if
\begin{itemize}
%\item  $\iota^F_c(m)F(\iota^m_c)=Id_{F(c\cdot m)}$.
\item  $F(\iota^m_c)\iota^F_c(m)=Id_{F(m\cdot c)}$.
  \end{itemize}
\end{definition}

\begin{lemma}\label{pairtotrace}
If $(F, m)$ is a (stable) central pair then $F(m \cdot -)$ is a (symmetric) 2-trace.
\end{lemma}
\begin{proof}
Define the structure of a 2-trace on $F(m \cdot -)$, i.e., an isomorphism $$\iota_c: F(m\cdot (-\ot c))\simeq F(m\cdot(c\ot -))$$ via the chain of isomorphisms:
\begin{align*}
F(m\cdot (-\ot c))&\rightarrow F((m\cdot -)\cdot c)\\
&\rightarrow F(c\cdot(m\cdot -))\\
&\rightarrow F((c\cdot m)\cdot -)\\
&\rightarrow F((m\cdot c)\cdot -)\\
&\rightarrow F(m\cdot(c\ot -))
\end{align*}

For the symmetry condition consult the following commutative diagram with all arrows being the obvious isomorphisms:

$$
\xymatrix{
F(m\cdot c)\ar[r]\ar[dr]\ar[dd] & F(m\cdot (1\ot c))\ar[d]\\
& F((m\cdot 1)\cdot c)\ar[d]\\
F(c\cdot m)\ar[r]\ar[dr]\ar[dd] & F(c\cdot(m\cdot 1))\ar[d]\\
& F((c\cdot m)\cdot 1)\ar[d]\\
F(m\cdot c)\ar[r]\ar[dr] & F((m\cdot c)\cdot 1)\ar[d]\\
& F(m\cdot (c\otimes 1))\\
}
$$ then the first column composes to $Id$ by stability, and the second column composes to $\iota_c(1)$ by definition, the claim follows.
\end{proof}

This  shows that a stable central pair gives us a symmetric 2-trace and therefore by Propositions \ref{trace1} and \ref{trace2} produces cyclic and cocyclic objects from algebras and coalgebras.  More precisely, $F(m\cdot A^{\ot \bullet+1})$ and $F(m\cdot C^{\ot \bullet+1})$ are cyclic and cocyclic objects, for $A$ an algebra and $C$ a coalgebra respectively.  Let us write out the cyclic map for these cases.  Roughly speaking, i.e., ignoring the associativity isomorphisms we have $$\tau: F(m\cdot A^{\ot n+1})\xrightarrow{\iota^F_A(m\cdot A^{\ot n})} F(A\cdot m\cdot A^{\ot n})\xrightarrow{F(\iota^m_A\cdot Id^{\ot n})} F(m\cdot A^{\ot n+1})$$ while $$\tau: F(m\cdot C^{\ot n+1})\xrightarrow{F((\iota^m_C)^{-1}\cdot Id^{\ot n})} F(C\cdot m\cdot C^{\ot n})\xrightarrow{(\iota^F_C)^{-1}(m\cdot C^{\ot n})} F(m\cdot C^{\ot n+1}),$$ so that in the algebra case $\tau$ moves ``last to first" while in the coalgebra case it does the opposite.

\subsubsection{Stable central contrapairs.}
Let us now mirror the above discussion for the contravariant case.

\begin{definition}\label{stablecontrapair}
  Let $(\mathcal{C}, \ot)$ be a monoidal category,  and  $\mc$ a $\Cc$-bimodule category.
    Let $F\in Fun (\mc_{op}, \vect)$ and $m\in \mc$.    The pair $(F, m)$ is called a stable central contrapair if
  $(F,m')$ is a stable central pair for $\Cc_{op}$.
\end{definition}

We immediately obtain:

\begin{lemma}\label{contrapairtocontratrace}
If $(F, m)$ is a stable central contrapair then $F(m \cdot -)$ is a symmetric $2$-contratrace.
\end{lemma}

Thus as above, $F(m\cdot A^{\ot \bullet+1})$ and $F(m\cdot C^{\ot \bullet+1})$ are cocyclic and cyclic objects, for $A$ an algebra and $C$ a coalgebra respectively.  Ignoring the associativity isomorphisms we have $$\tau: F(m\cdot A\ot A^{\ot n})\rightarrow F(A\cdot m\cdot A^{\ot n})\rightarrow F(m\cdot A^{\ot n}\ot A)$$ while $$\tau: F(m\cdot C^{\ot n}\ot C)\rightarrow F(C\cdot m\cdot C^{\ot n})\rightarrow F(m\cdot C\ot C^{\ot n}),$$ so that in the coalgebra case $\tau$ moves ``last to first" while in the algebra case it does the opposite.

\section{The monoidal category of left modules over a Hopf algebra}

In this section we  apply our results from Section \ref{moncat2tr} to the monoidal category of left modules over a Hopf algebra $H$. Our aim is  to construct a symmetric 2-trace via a stable central pair.
The idea was sketched in Section \ref{motivation} for module coalgebras, here we briefly recap for module algebras.  The only difference is that in the definition of $\tau$ the order of what gets used first: the centrality of the functor or the centrality of the element gets reversed.

Let us consider a simpler version of what we want, namely $\hmod_{fd}$ which is a rigid monoidal category. Using the  rigid structure we have  the following isomorphism

\begin{equation}\label{adj1}
  \Hom_H(1, -\ot V)\xrightarrow\sim \Hom_H(1, {^{**}V}\ot -),
\end{equation}

and furthermore

$$\Hom_H(1, -)\in \Zc_{\hmod_{fd}} Fun(\tthmod_{fd}, \vect).$$  If in addition $$M\in\Zc_{\hmod_{fd}}(\tthmod_{fd}),$$ in particular we have \begin{equation}\label{centre}{^{**}-}\ot M   \xrightarrow\sim M\ot -,\end{equation} then we can make a cyclic map $\tau$ as follows.

Consider an (algebra) object $A$ in $\hmod_{fd}$ then we obtain

$$\tau: \Hom_{H}(1, M\ot A^{\ot n+1})\xrightarrow\sim  \Hom_{H}(1, {^{**}A}\ot M\ot A^{\ot n}) \xrightarrow\sim \Hom_{H}(1, M\ot A\ot A^{\ot n}),$$ where we first used \eqref{adj1} and then \eqref{centre} thus \emph{sliding the last copy of $A$ to the front and past $M$}.
Of course we only need the algebra structure to define the simplicial structure, the map $\tau$ above does not need it. The resulting structure on $\Hom_{H}(1, M\ot A^{\ot n+1})$ is that of a cyclic module, provided that $$\tau_0:\Hom_{H}(1, M\ot A)\xrightarrow\sim \Hom_{H}(1, M\ot A)$$ is the identity map.  If the latter condition is dropped then the result is a paracyclic module.

We note that the conditions on $M$ as outlined above are equivalent to the $1$-stable $YD_1$ condition, see below for details.

\subsection{The covariant theory for $\hmod$}\label{2traceforhmod}

In this subsection we consider  the monoidal category $\hmod$ and show that if $M$ is a left-right $\mathcal{YD}_1$ module, then the  functor  $\Hom_H(1,-)$ paired with $M$ forms a central pair $\left(\Hom_H(1, -), M\right)$ for a suitable bimodule category, namely $\mc=\tinvmod$. Furthermore if $M$ is $1$-stable, then $\left(\Hom_H(1, -), M\right)$ is a stable central pair.

Recall that to prove  that $\left(\Hom_H(1, -), M\right)$ is a stable central pair for $\hmod$ and its bimodule category $\tinvmod$,   we need to show that

\begin{itemize}
  \item $\Hom_H(1, -)\in \Zc_\hmod Fun(\tinvmod,\vect)$.
  \item If $M\in\mathcal{YD}_1$ then $M\in \Zc_\hmod(\tinvmod)$.
  \item $1$-stability of $M$ ensures that $\tau_0=\Id$ and therefore the stability of the  central pair.
\end{itemize}

The second point is the content of Theorem \ref{ydcenter}.  Now to address the first point.

If  $A, B\in \hmod$, then the  left $H$-module map  $f\in \Hom_H(1,A\ot B)$ is equivalent to the data of an $H$-invariant element $a\ot b\in A\ot B$.  More precisely,
$$h\ps{1}a\ot h\ps{2} b=\varepsilon(h)a\ot b,$$ for all $h\in H$.  Note that we write $a\ot b$ when we actually mean a sum of such elements in $(A\ot B)^H$.

\begin{lemma}\label{cov-H-mod}
  Let $A, B\in \hmod$ and $a\ot b\in A\ot B$. Then $$h\ps{1}a\ot h\ps{2} b=\varepsilon(h)a\ot b \quad\Leftrightarrow\quad S(h)a\ot b= a\ot hb.$$
\end{lemma}
\begin{proof}
If $S(h)a\ot b=a\ot hb$ then
$\varepsilon(h)a\ot b=h\ps{1}S(h\ps{2})a\ot b=h\ps{1}a\ot h\ps{2}b.
 $
  Conversely if $h\ps{1}a\ot h\ps{2}b=\varepsilon(h) a\ot b$ then
  \begin{align*}
    S(h)a\ot b&= S(h\ps{1}\varepsilon(h\ps{2}))a\ot b\\
    &= S(h\ps{1})\varepsilon(h\ps{2})a\ot b\\
    &= S(h\ps{1})h\ps{2}a\ot h\ps{3}b\\
    &= \varepsilon(h\ps{1})a\ot h\ps{2}b\\
    &= a\ot \varepsilon(h\ps{1})h\ps{2}b\\
    &= a\ot hb.
  \end{align*}
\end{proof}

For the sake of reducing notational clutter, let us, for an element $B\in\hmod$, denote by ${^\# B}$ what was until now called $B^{\#^{-1}}$.

\begin{proposition}\label{homcentral}
$$\Hom_H(1, -)\in \Zc_\hmod Fun(\tinvmod, \vect).$$
\end{proposition}

\begin{proof}

Since
$$
  S(h)a\ot b= a\ot hb\quad\forall h\in H\quad\Leftrightarrow\quad S^{-1}(h)b\ot a=b\ot ha\quad\forall h\in H,
$$

so in view of the  Lemma \ref{cov-H-mod}, we see that $$a\ot b\in(A\ot B)^H\quad\Leftrightarrow\quad b\ot a\in({^\#B}\ot A)^H,$$
i.e,
 \begin{equation}\label{homcent} a\ot b\longmapsto b\ot a\end{equation} obviously gives rise to the following  natural isomorphisms

$$
  \Hom_H(1, -\ot B)\xrightarrow\sim \Hom_H(1,{^\#B}\ot -);
$$

it is not hard to check the rest given that the map itself is very simple.

\end{proof}

On to the third point: stability.  We need to check that  $\tau_0=Id$ if $M$ is $1$-stable.  Note that  strictly speaking $\tau_0$ depends on wether we want to use the theory on algebras or coalgebras, otherwise we might need its inverse. Yet $\tau_0=Id$ if and only if $\tau_0^{-1}=Id$; this is not surprising as the notion of a stable central pair doesn't depend on what you intend to use it for.

\begin{lemma}
If $M$ is $1$-stable then $\tau_0=Id$.
\end{lemma}

\begin{proof}
Using \eqref{homcent} followed by \eqref{mcent} we see that for $m\ot v\in(M\ot V)^H$ we have $\tau_0(m\ot v)=m\ns{0}\ot m\ns{1}v$ which is $S(m\ns{1})m\ns{0}\ot v$ by Lemma \ref{cov-H-mod}, and the latter is $m\ot v$ by the $1$-stability of $M$ in view of Lemma \ref{stability}.
\end{proof}

Thus $(\Hom_H(1,-), M)$ is a stable central pair.  As usual let $A$ be an algebra and $C$ a coalgebra. We write out the formulas for the cyclic and cocyclic module structures.

For $C_n=\Hom_H(1, M\ot A^{\ot n+1})$:
\begin{align*}
  \delta_i(m\ot a_0\ot \cdots \ot a_n)&=m\ot a_0\ot \cdots \ot a_i a_{i+1}\ot \cdots \ot a_n,\\
  \delta_n(m\ot a_0\ot \cdots \ot a_n)&=m\ns{0}\ot (m\ns{1}a_n) a_0\ot a_1\ot \cdots \ot a_{n-1},  \\
  \sigma_i(m\ot a_0\ot \cdots \ot a_n)&=m\ot a_0\ot \cdots \ot a_i \ot 1\ot \cdots \ot a_n,\\
  \tau_n(m\ot a_0\ot \cdots \ot a_n)&=m\ns{0}\ot m\ns{1}a_n\ot a_0\ot \cdots \ot a_{n-1}.
\end{align*}

For $C^n=\Hom_H(1, M\ot C^{\ot n+1})$:
\begin{align*}
  \delta_i(m\ot c_0\ot \cdots \ot c_{n-1})&=m\ot c_0\ot \cdots \ot c_i\ps{1}\ot c_i\ps{2}\ot \cdots \ot c_{n-1},\\
  \delta_n(m\ot c_0\ot \cdots \ot c_{n-1})&=m\ns{0}\ot  c_0\ps{2}\ot c_1\ot \cdots \ot c_{n-1}\ot S(m\ns{1})c_0\ps{1}, \\
  \sigma_i(m\ot c_0\ot \cdots \ot c_{n+1})&=m\ot c_0\ot \cdots \ot \varepsilon(c_{i+1})\ot \cdots \ot c_{n+1},\\
  \tau_n(m\ot c_0\ot \cdots \ot c_n)&=m\ns{0}\ot c_1\ot \cdots \ot c_{n}\ot S(m\ns{1})c_0.
\end{align*}

\subsection{The contravariant theory for $\hmod$}

Here we redo the previous section for the case of $\Hom_H(-,1)$.  More precisely, we show
 that if $M$ is a left-right $\mathcal{YD}_{-1}$ module ($\mathcal{AYD}$ module), then the  functor  $\Hom_H(-,1)$ paired with $M$ forms a central contrapair $\left(\Hom_H(-, 1), M\right)$ for $\mc=\thmod$. Furthermore if $M$ is $0$-stable (classically stable), then $\left(\Hom_H(-, 1), M\right)$ is a stable central contrapair.

We will need a characterization of $H$ module maps from $A\ot B$ to the monoidal unit $k$.
\begin{lemma}\label{H-maps}
Let $H$ be a Hopf algebra over a field $k$, $A, B\in \hmod$ and $f: A\ot B\rightarrow k$ a $k$-linear map. Then $f(h\ps{1}a\ot h\ps{2}b)=\varepsilon(h) f(a\ot b)$ if and only if $f(ha\ot b)=f(a\ot S(h)b)$.
\end{lemma}
\begin{proof}
The computation is similar to that of Lemma \ref{cov-H-mod}. If $f(ha\ot b)=f(a\ot S(h)b)$ then
$f(h\ps{1}a\ot h\ps{2}b)=f(a\ot S(h\ps{1})h\ps{2}b)=\varepsilon(h) f(a\ot b).
 $
  Conversely if $f(h\ps{1}a\ot h\ps{2}b)=\varepsilon(h) f(a\ot b)$ then
  \begin{align*}
    f(a\ot S(h)b)&= f(a\ot S(\varepsilon(h\ps{1})h\ps{2})b)\\
    &= \varepsilon(h\ps{1})f(a\ot S(h\ps{2})b)\\
    &= f(h\ps{1}a\ot h\ps{2} S(h\ps{3})b)\\
    &= f(h\ps{1}a\ot \varepsilon(h\ps{2})b)\\
    &= f(ha\ot b).
  \end{align*}

\end{proof}
\begin{proposition}\label{homcentral1}
$$\Hom_H(-, 1)\in \Zc_{\hmod_{op}} Fun(\thmod_{op}, \vect).$$
\end{proposition}

\begin{proof}
This boils down to observing that for $f\in\Hom_H(A^\#\ot B,1)$ we have $\gamma f\in\Hom_H(B\ot A,1)$ where $\gamma f(b\ot a)=f(a\ot b)$.  This is an immediate consequence of Lemma \ref{H-maps}.

\end{proof}

We recall that by Theorem \ref{ydcenter}, if $M\in\mathcal{AYD}$, i.e., $M\in\mathcal{YD}_{-1}$ then $M\in\Zc_{\hmod}(\thmod)$ and the latter is equivalent to $\Zc_{\hmod_{op}}(\thmod_{op})$.  So all that remains is to investigate the stability condition needed for $\tau_0=Id$.

\begin{lemma}
If $M$ is $0$-stable (classically stable) then $\tau_0=Id$.
\end{lemma}

\begin{proof}
Let $f\in\Hom_H(M\ot V,1)$, then $\tau_0 f(m\ot v)=f(m\ns{0}\ot S(m\ns{1})v)=f(m\ns{1}m\ns{0}\ot v)=f(m\ot v)$.
\end{proof}

Thus $(\Hom_H(-,1), M)$ is a stable central contrapair.  As usual let $A$ be an algebra and $C$ a coalgebra. We write out the formulas for the cocyclic and cyclic module structures.

For $C^n=\Hom_H(M\ot A^{\ot n+1},1)$:
\begin{align*}
  \delta_if(m\ot a_0\ot \cdots \ot a_n)&=f(m\ot a_0\ot \cdots \ot a_i a_{i+1}\ot \cdots \ot a_n),\\
  \delta_nf(m\ot a_0\ot \cdots \ot a_n)&=f(m\ns{0}\ot (m\ns{1}a_n) a_0\ot a_1\ot \cdots \ot a_{n-1}),  \\
  \sigma_if(m\ot a_0\ot \cdots \ot a_n)&=f(m\ot a_0\ot \cdots \ot a_i \ot 1\ot \cdots \ot a_n),\\
  \tau_nf(m\ot a_0\ot \cdots \ot a_n)&=f(m\ns{0}\ot m\ns{1}a_n\ot a_0\ot \cdots \ot a_{n-1}).
\end{align*}

For $C_n=\Hom_H(M\ot C^{\ot n+1},1)$:
\begin{align*}
  \delta_if(m\ot c_0\ot \cdots \ot c_{n-1})&=f(m\ot c_0\ot \cdots \ot c_i\ps{1}\ot c_i\ps{2}\ot \cdots \ot c_{n-1}),\\
  \delta_nf(m\ot c_0\ot \cdots \ot c_{n-1})&=f(m\ns{0}\ot  c_0\ps{2}\ot c_1\ot \cdots \ot c_{n-1}\ot S(m\ns{1})c_0\ps{1}), \\
  \sigma_if(m\ot c_0\ot \cdots \ot c_{n+1})&=f(m\ot c_0\ot \cdots \ot \varepsilon(c_{i+1})\ot \cdots \ot c_{n+1}),\\
  \tau_nf(m\ot c_0\ot \cdots \ot c_n)&=f(m\ns{0}\ot c_1\ot \cdots \ot c_{n}\ot S(m\ns{1})c_0).
\end{align*}

In this paper we have investigated the four (co)homology theories that arise naturally in the consideration of the monoidal category of left $H$-modules.  These come from the considerations of the covariant and the contravariant theories in the sense of their behavior with respect to maps of (co)algebras.    If we consider the contravariant theory of the algebra case, we recover the type $A$ cohomology theory of \cite{hkrs2} on the nose.  By considering the covariant theory of the coalgebra case we obtain a different cohomology theory than that of \cite{hkrs2}; this explains the need for \emph{new} coefficients of opposite ``charge" than AYD.   The type $C$ theory which generalizes Connes-Moscovici Hopf cyclic cohomology \cite{CM98} is actually obtained from the $2$-contratrace giving the type $A$ theory by considering its predual which is a $2$-trace. The other two possibilities considered are both homology theories, one requiring AYDs and the other ``anti" AYDs.

Our explicit calculations do not extend to the type $B$ theory of \cite{hkrs2} which is a contravariant cohomology theory for $H$ \emph{comodule} algebras.  However we point out that our machinery can be applied to the monoidal category ${^H\mc}$ of left comodules over $H$.  In that case the type $B$ theory is a straightforward consequence, though with modifications.  

Let us summarize. If we are given a rigid monoidal category $\Cc$ then there is a covariant cyclic theory with coefficients in $\Zc_\Cc\Cc^{**}$ that turns algebras into cyclic modules and coalgebras into cocyclic modules. There is also a contravariant cyclic theory with coefficients in $\Zc_\Cc{^{**}\Cc}$ that turns algebras into cocyclic modules and coalgebras into cyclic modules.  In the above $^{**}$ is the functor that sends $c\in\Cc$ to $c^{**}$. Note that the coefficients need to be more than just central in a correct bimodule category, they have to be stable as well.  If the category is not rigid \emph{anymore}, such as was the case of the general $\hmod$ with infinite dimensional representations allowed, we can still proceed: we would need a replacement for $(-)^{**}$, such as $\#$ was in the case of $\hmod$.

%%%%%%%%%%%%%%%%%%%%%%%%%%%%%%%%%%%%%%%%%%%%%%%%%%%%%%%%%%%%%%%%%%%%%%%%%%%%%%%%%%%%%%%%%%%%%%%%%%%%%%

\bigskip

\noindent Department of Mathematics and Statistics,
University of Windsor, 401 Sunset Avenue, Windsor, Ontario N9B 3P4, Canada

\noindent\emph{E-mail address}:
\textbf{mhassan@uwindsor.ca}
\medskip

\noindent Department of Mathematics,
Middlesex College,
The University of Western Ontario,

London, Ontario, N6A 5B7,
Canada.

\noindent\emph{E-mail address}:
\textbf{masoud@uwo.ca}
\medskip

\noindent Department of Mathematics and Statistics,
University of Windsor, 401 Sunset Avenue, Windsor, Ontario N9B 3P4, Canada

\noindent\emph{E-mail address}:
\textbf{ishapiro@uwindsor.ca}
\medskip

\end{document}